\documentclass[11pt]{amsart}

\usepackage{amssymb}
\usepackage{amsthm}
\usepackage{amsmath}
\usepackage{amsfonts}
\usepackage{mathrsfs}

\vfuzz2pt \DeclareMathOperator{\RE}{Re}

\DeclareMathOperator{\Psl}{PSL} 

 \DeclareMathOperator{\Ad}{Ad}
 \DeclareMathOperator{\sign}{sign}
\DeclareMathOperator{\TRR}{Tr} \DeclareMathOperator{\MOD}{mod}
\DeclareMathOperator{\vol}{vol} \DeclareMathOperator{\spec}{spec}
\DeclareMathOperator{\degg}{deg}
\DeclareMathOperator{\II}{i}
\DeclareMathOperator{\IDD}{id}

\newtheorem{theorem}{Theorem}
\newtheorem{corollary}[theorem]{Corollary}

\newtheorem{definition}[theorem]{Definition}

\newenvironment{teorem}[2][Theorem]{\begin{trivlist}
\item[\hskip \labelsep {\bfseries #1}\hskip \labelsep {\bfseries #2}]}{\end{trivlist}}

\newenvironment{rem}[2][Acknowledgment.]{\begin{trivlist}
\item[\hskip \labelsep {\bfseries #1}\hskip \labelsep {\bfseries #2}]}{\end{trivlist}}

\numberwithin{equation}{section}
\numberwithin{theorem}{section}

\newcommand{\set}[1]{\left\{#1\right\}}
\newcommand{\abs}[1]{\left\vert#1\right\vert}
\newcommand{\br}[1]{\left(#1\right)}
\newcommand{\SqBr}[1]{\left[#1\right]}

\begin{document}

\title[Order of zeta functions]{Order of zeta functions for compact even-dimensional symmetric spaces}
\author{M. Avdispahi\'c, D\v z. Gu\v si\'c and D. Kamber}

\address{University of Sarajevo, Department of Mathematics, Zmaja od Bosne
35, 71000 Sarajevo, Bosnia and Herzegovina}
\email{\textbf{mavdispa@pmf.unsa.ba}}

\address{University of Sarajevo, Department of Mathematics, Zmaja od Bosne
35, 71000 Sarajevo, Bosnia and Herzegovina}
\email{\textbf{dzenang@pmf.unsa.ba}}

\address{University of Sarajevo, Department of Mathematics, Zmaja od Bosne
35, 71000 Sarajevo, Bosnia and Herzegovina}
\email{\textbf{dinakamber@pmf.unsa.ba}}



\maketitle

\begin{abstract}
Some zeta functions which are naturally attached to the locally homogeneous vector bundles over compact locally symmetric spaces of rank one are investigated. We prove that such functions can be expressed in terms of entire functions whose order is not larger than the dimension of the corresponding compact, even-dimensional, locally symmetric space.
\end{abstract}

\section{Introduction}

Let $Y$ be a compact $n-$ dimensional (n even), locally symmetric Riemannian manifold with strictly negative sectional curvature.

In \cite{Bunke}, authors derived the properties of certain zeta functions canonically associated with the geodesic flow of $Y$. Motivated by the fact that the classical Selberg zeta function \cite{Selberg} is an entire function of order two and following Park's method \cite{Park} on hyperbolic manifolds with cusps, Avdispahi\'c and Gu\v si\'c \cite{AG} derived an analogous result for the zeta functions described in \cite{Bunke}. The main purpose of this paper is to give yet another proof of the result \cite{AG} based on Pavey's approach \cite{Pav} in the quartic fields setting (see also, \cite{DeitmarPavey}).

We found the second proof worth presenting because of the following facts.

Note that both approaches appear in literature. More importantly, their ingredients as well as the results they yield, are usually used later on, either to obtain appropriate estimates on the number of singularities or to derive approximate formulas for the logarithmic derivative of zeta functions (see, e.g., \cite{Randol2}, \cite{Pav}). These results are further exploited to achieve more refined error terms in the prime geodesic theorem. Various authors treat the compact hyperbolic Riemann surface case \cite{Randol1}, \cite{Hejhal1}, quartic fields case \cite{Pav}, \cite{DeitmarPavey}, higher dimensional manifolds case \cite{AG0}, \cite{Park}.

The differences between the first and the second proof will be pointed out in the sequel.

\section{Notation and normalization}

Let $X$ be the universal covering of $Y$. We have $X=G^{'}/K^{'}$, where $G^{'}$ is the group of orientation preserving isometries of $X$ and $K^{'}$ is the stabilizer of some fixed base point $x_{0}\in X$. $G^{'}$ is a connected semisimple Lie group of real rank one and $K^{'}$ is maximal compact in $G^{'}$.

Let $\Gamma$ be a  discrete, co-compact, torsion-free subgroup of $G^{'}$ such that $Y=\Gamma\backslash G^{'}/K^{'}$.

Following \cite[p.~17]{Bunke}, we consider a finite covering group $G$ of $G^{'}$ such that the embedding $\Gamma\hookrightarrow G^{'}$ lifts to an embedding $\Gamma\hookrightarrow G$. We have $X=G/K$, $Y=\Gamma\backslash G/K$, where $K$ is the preimage of $K^{'}$ under the covering and $K$ is maximal compact in $G$.

Assume that $G$ is a linear group.

We denote the Lie algebras of $G$, $K$ by $\mathfrak{g}$, $\mathfrak{k}$, respectively. Let $\mathfrak{g}=\mathfrak{k}\oplus\mathfrak{p}$ be the Cartan decomposition with Cartan involution $\theta$.

A compact dual space $X_{d}$ of $X$ is given by $X_{d}=G_{d}/K$, where $G_{d}$ is the analytic subgroup of $GL\br{n,\mathbb{C}}$ corresponding to $\mathfrak{g}_{d}=\mathfrak{k}\oplus\mathfrak{p}_{d}$, $\mathfrak{p}_{d}=\II\mathfrak{p}$. Following \cite[p.~18]{Bunke}, we normalize the metric on $X_{d}$ such that the multiplication by $\II$ induces an isometry between $\mathfrak{p}\cong T_{x_{0}}X$ and $\mathfrak{p}_{d}\cong T_{x_{0}}X_{d}$. We normalize the $\Ad\br{G}$-invariant inner product $\br{.,.}$ on $\mathfrak{g}$ to restrict to the metric on $T_{x_{0}}X\cong\mathfrak{p}$.

Let $\mathfrak{a}$ be a maximal abelian subspace of $\mathfrak{p}$. Then, the dimension of $\mathfrak{a}$ is one. Let $M$ be the centralizer of $A=\exp\br{\mathfrak{a}}$ in $K$ with Lie algebra $\mathfrak{m}$. We have $SX=G/M$, $SY=\Gamma\backslash G/M$, where $SX$ and $SY$ denote the unit sphere bundles of $X$ and $Y$, respectively.

Denote by $\mathfrak{g}_{\mathbb{C}}$, $\mathfrak{a}_{\mathbb{C}}$, etc. the complexifications of $\mathfrak{g}$, $\mathfrak{a}$, etc.

We fix the set of positive roots $\Phi^{+}\br{\mathfrak{g},\mathfrak{a}}$ of $\br{\mathfrak{g},\mathfrak{a}}$ and define

\[\rho=\frac{1}{2}\sum_{\alpha\in\Phi^{+}\br{\mathfrak{g},\mathfrak{a}}}\dim\br{\mathfrak{n}_{\alpha}}\alpha\in\mathfrak{a}_{\mathbb{C}}^{*},\]
\newline
where $\mathfrak{n}_{\alpha}$ denotes the root space of $\alpha$.

By Propositions 1.1 and 1.2 \cite[pp.~20-23]{Bunke}, for every $\sigma\in\hat{M}$ there exists an element $\gamma\in R\br{K}$ such that $i^{*}\br{\gamma}=\sigma$, where $i^{*}:R\br{K}\rightarrow R\br{M}$ is the restriction and $R\br{K}$, $R\br{M}$ are the representation rings over $\mathbb{Z}$ of $K$, $M$, respectively. Following \cite[p.~27]{Bunke}, we associate to $\gamma=\sum a_{i}\gamma_{i}$, ($a_{i}\in\mathbb{Z}$, $\gamma_{i}\in\hat{K}$), $\mathbb{Z}_{2}$-graded homogeneous vector bundles $V\br{\gamma}=V\br{\gamma}^{+}\oplus V\br{\gamma}^{-}$ and $V_{d}\br{\gamma}=V_{d}\br{\gamma}^{+}\oplus V_{d}\br{\gamma}^{-}$ on $X$ and $X_{d}$, respectively, such that

\[V\br{\gamma}^{\pm}=G\times_{K}V_{\gamma}^{\pm},\]
\[V_{d}\br{\gamma}^{\pm}=G_{d}\times_{K}V_{\gamma}^{\pm},\]
\[V_{\gamma}^{\pm}=\bigoplus_{\sign\br{a_{i}}=\pm 1}\bigoplus_{m=1}^{\abs{a_{i}}}V_{\gamma_{i}},\]
\newline
where $V_{\gamma_{i}}$ is the representation space of $\gamma_{i}$. Furthermore, for a finite-dimensional unitary representation $\br{\chi,V_{\chi}}$ of $\Gamma$, we define $\mathbb{Z}_{2}$-graded vector bundle $V_{Y,\chi}\br{\gamma}=\Gamma\backslash\br{V_{\chi}\otimes V\br{\gamma}}$ on $Y$.

If $\mathfrak{t}$ is a Cartan subalgebra of $\mathfrak{m}$, then $\rho_{\mathfrak{m}}=\frac{1}{2}\sum_{\alpha\in\Phi^{+}\br{\mathfrak{m}_{\mathbb{C}},\mathfrak{t}}}\in\II\mathfrak{t}^{*}$ and the highest weight $\mu_{\sigma}\in\II\mathfrak{t}^{*}$ of $\sigma$ are defined (see, \cite[pp.~19-20]{Bunke}). Put

\[c\br{\sigma}=\abs{\rho}^{2}+\abs{\rho_{\mathfrak{m}}}^{2}-\abs{\mu_{\sigma}+\rho_{\mathfrak{m}}}^{2}.\]
\newline
Note that the norms are defined by the complex bilinear extension to $\mathfrak{g}_{\mathbb{C}}$ of $\br{.,.}$.

Let $\Omega$ be the Casimir element of the complex universal enveloping algebra $\mathcal{U}\br{\mathfrak{g}}=\mathcal{U}\br{\mathfrak{g}_{d}}$ of $\mathfrak{g}$, which is also fixed by $\br{.,.}$. $\Omega$ acts in a natural way on sections of the bundles $V\br{\gamma}$ and $V_{d}\br{\gamma}$. Moreover, it acts as a $G$-invariant differential operator on $C^{\infty}\br{X,V\br{\gamma}}$ and hence it descends to $C^{\infty}\br{Y,V_{Y,\chi}\br{\gamma}}$. Therefore, it makes sense to define the operators (see, \cite[p.~28]{Bunke})
\[A_{d}\br{\gamma,\sigma}^{2}=\Omega+c\br{\sigma}:C^{\infty}\br{X,V_{d}\br{\gamma}}\rightarrow C^{\infty}\br{X,V_{d}\br{\gamma}},\]
\[A_{Y,\chi}\br{\gamma,\sigma}^{2}=-\Omega-c\br{\sigma}:C^{\infty}\br{Y,V_{Y,\chi}\br{\gamma}}\rightarrow C^{\infty}\br{Y,V_{Y,\chi}\br{\gamma}}.\]
Put
\[\delta=\frac{1}{2}\sum_{\alpha\in\Phi^{+}\br{\mathfrak{g}_{\mathbb{C}},\mathfrak{h}_{\mathbb{C}}}}\alpha,\]
\newline
where $\mathfrak{h}_{\mathbb{C}}=\mathfrak{t}_{\mathbb{C}}\oplus\mathfrak{a}_{\mathbb{C}}$, $\mathfrak{t}$ is a maximal abelian subalgebra of $\mathfrak{m}$ and $\Phi^{+}\br{\mathfrak{g}_{\mathbb{C}},\mathfrak{h}_{\mathbb{C}}}$ is the set of positive roots satisfying condition that $\alpha_{|\mathfrak{a}}\in\Phi^{+}\br{\mathfrak{g},\mathfrak{a}}$ implies $\alpha\in\Phi^{+}\br{\mathfrak{g}_{\mathbb{C}},\mathfrak{h}_{\mathbb{C}}}$ for $\alpha\in\Phi\br{\mathfrak{g}_{\mathbb{C}},\mathfrak{h}_{\mathbb{C}}}$. Define $\rho_{\mathfrak{m}}=\delta-\rho$ and the root vector $H_{\alpha}\in\mathfrak{a}$ for $\alpha\in\Phi^{+}\br{\mathfrak{g},\mathfrak{a}}$ by

\[\lambda\br{H_{\alpha}}=\frac{\br{\lambda,\alpha}}{\br{\alpha,\alpha}},\,\,\forall\lambda\in\mathfrak{a}^{*}.\]
\newline
We adopt the following two definitions from \cite{Bunke}.
\newline
\begin{definition}\label{definition2.1}
\cite[p.~47, Def. 1.13]{Bunke} Let $\alpha\in\Phi^{+}\br{\mathfrak{g},\mathfrak{a}}$ be the long root and $T=\abs{\alpha}$. For $\sigma\in\hat{M}$ we define $\epsilon_{\sigma}\in\set{0,\frac{1}{2}}$ by

\[\epsilon_{\sigma}=\frac{\abs{\rho}}{T}+\varepsilon_{\alpha}\br{\sigma}\,\MOD\,\mathbb{Z}\]
\newline
and the lattice $L\br{\sigma}\subset\mathbb{R}\cong\mathfrak{a}^{*}$ by $L\br{\sigma}=T\br{\epsilon_{\sigma}+\mathbb{Z}}$. Finally, for $\lambda\in\mathfrak{a}_{\mathbb{C}}^{*}\cong\mathbb{C}$ we set
\[P_{\sigma}\br{\lambda}=\prod_{\beta\in\Phi^{+}\br{\mathfrak{g}_{\mathbb{C}},\mathfrak{h}_{\mathbb{C}}}}\frac{\br{\lambda+\mu_{\sigma}+\rho_{\mathfrak{m}},\beta}}{\br{\delta,\beta}}.\]
\end{definition}\
\newline
Note that $\varepsilon_{\alpha}\br{\sigma}\in\set{0,\frac{1}{2}}$ for $\alpha\in\Phi^{+}\br{\mathfrak{g},\mathfrak{a}}$ is defined by $\exp\br{2\pi\II\varepsilon_{\alpha}\br{\sigma}}=\sigma\br{\exp\br{2\pi\II H_{\alpha}}}\in\set{\pm 1}$ (see, \cite[p.~40]{Bunke}). Also, the fact that $\alpha\in\Phi^{+}\br{\mathfrak{g},\mathfrak{a}}$ is the long root means that $\Phi^{+}\br{\mathfrak{g},\mathfrak{a}}=\set{\alpha}$ or $\Phi^{+}\br{\mathfrak{g},\mathfrak{a}}=\set{\frac{\alpha}{2},\alpha}$.

\begin{definition}\label{definition2.2}
\cite[p.~49, Def. 1.17]{Bunke} Let $\sigma\in\hat{M}$. Then, $\gamma\in R\br{K}$ is called $\sigma$-admissible if $i^{*}\br{\gamma}=\sigma$ and $m_{d}\br{s,\gamma,\sigma}=P_{\sigma}\br{s}$ for all $0\leq s\in L\br{\sigma}$.
\end{definition}\
\newline
The multiplicity $m_{d}\br{s,\gamma,\sigma}$ is the weighted dimension of eigenspace of $A_{d}\br{\gamma,\sigma}$, i.e., $m_{d}\br{s,\gamma,\sigma}=\TRR E_{A_{d}\br{\gamma,\sigma}}\br{\set{s}}$, $s\in\mathbb{C}$, where $E_{A}\br{.}$ is the family of spectral projections of some normal operator $A$. Similarly, we introduce $m_{\chi}\br{s,\gamma,\sigma}=\TRR E_{A_{Y,\chi}\br{\gamma,\sigma}}\br{\set{s}}$, $s\in\mathbb{C}$ (see, \cite[p.~30]{Bunke}).

Through the rest of the paper we shall assume that the metric on $Y$ is normalized such that the sectional curvature varies between $-1$ and $-4$. Then, $T\geq 1$ (see, \cite[pp.~150f]{Bunke}).

\section{Zeta functions of Selberg and Ruelle}

Denote by $C\Gamma$ the set of conjugacy classes of $\Gamma$.

$G/\Gamma$ is a compact space. Hence, $\Gamma$ does not contain parabolic elements. On the other hand, being torsion-free, $\Gamma$ contains no nontrivial eliptic elements. In other words, every $\gamma\in\Gamma$, $\gamma\neq 1$ is hyperbolic.

As usual, we may assume that a hyperbolic element $g\in G$ has the form $g=am\in A^{+}M$, where $A^{+}=\exp\br{\mathfrak{a}^{+}}$ and $\mathfrak{a}^{+}$ is the positive Weyl chamber in $\mathfrak{a}$ (see e.g., \cite{GangoliLength}, \cite{GangoliZetaCompact}).

Let $\br{\sigma,V_{\sigma}}$ and $\br{\chi,V_{\chi}}$ be some finite-dimensional unitary representations of $M$ and $\Gamma$, respectively and $V_{\chi}\br{\sigma}=\Gamma\backslash\br{G\times_{M}V_{\sigma}\otimes V_{\chi}}$ corresponding vector bundle. We introduce $\varphi_{\chi,\sigma}$ by (see, \cite[p.~95]{Bunke})

\[\varphi_{\chi,\sigma}:\mathbb{R}\times SY\ni\br{t,\Gamma gM}\rightarrow \Gamma g\exp\br{-tH}M\in SY\]
\newline
for some unit vector $H\in\mathfrak{a}^{+}$.

It is well known fact that free homotopy classes of closed paths on $Y$ are in a natural one-to-one correspondence with the set $C\Gamma$. In our situation, if $\SqBr{1}\neq\SqBr{g}\in C\Gamma$, the corresponding closed orbit $c$ is given by

\[c=\set{\Gamma g^{'}\exp\br{-tH}M\,|\,t\in\mathbb{R}},\]
\newline
where $g^{'}$ is chosen such that $\br{g^{'}}^{-1}gg^{'}=ma\in MA^{+}$. The length $l\br{c}$ of $c$ is given by $l\br{c}=\abs{\log a}$. The lift of $c$ to $V_{\chi}\br{\sigma}$ induces the monodromy $\mu_{\chi,\sigma}\br{c}$ on the fibre over $\Gamma g^{'}M$ as follows

\[\mu_{\chi,\sigma}\br{c}\br{\SqBr{g^{'},v\otimes w}}=\SqBr{g^{'},\sigma\br{m}v\otimes\chi\br{g}w}.\]
\newline
Now, for $s\in\mathbb{C}$, $\RE\br{s}>2\rho$, the Ruelle zeta function for the flow $\varphi_{\chi,\sigma}$ is defined by the infinite product

\[Z_{R,\chi}\br{s,\sigma}=\prod_{c\,\,prime}\dim\br{1-\mu_{\chi,\sigma}\br{c}e^{-sl\br{c}}}^{\br{-1}^{n-1}},\]
\newline
where a closed orbit $c$ through $y\in SY$ is called prime if $l\br{c}$ is the smallest time such that $\varphi\br{l\br{c},y}=y$.

Recall that the Anosov property of $\varphi$ means the existence of a $d\varphi$-invariant splitting
\[TSY=T^{s}SY\oplus T^{0}SY\oplus T^{u}SY,\]
\newline
where $T^{0}SY$ consists of vectors tangential to the orbits, while the vectors in $T^{s}SY$ $\br{T^{u}SY}$ shrink (grow) exponentially with respect to the metric as $t\rightarrow\infty$, when transported with $d\varphi$. In our case, the splitting is given by (see, \cite[p.~97]{Bunke})
\[TSY\cong\Gamma\backslash G\times_{M}\br{\bar{\mathfrak{n}}\oplus\mathfrak{a}\oplus\mathfrak{n}},\]
\newline
where $\bar{\mathfrak{n}}=\theta\mathfrak{n}$, $\mathfrak{n}=\sum_{\alpha\in\Phi^{+}\br{\mathfrak{g},\mathfrak{a}}}\mathfrak{n}_{\alpha}$. Now, the monodromy $P_{c}$ in $TSY$ of $c$ decomposes as follows
\[P_{c}=P_{c}^{s}\oplus\IDD\oplus P_{c}^{u}.\]
\newline
Finally, for $s\in\mathbb{C}$, $\RE\br{s}>\rho$, the Selberg zeta function for the flow $\varphi_{\chi,\sigma}$ is defined by the infinite product

\[Z_{S,\chi}\br{s,\sigma}=\prod_{c\,\,prime}\prod_{k=0}^{\infty}\det\br{1-\mu_{\chi,\sigma}\br{c}\otimes S^{k}\br{P_{c}^{s}}e^{-\br{s+\rho}l\br{c}}},\]
\newline
where $S^{k}$ denotes the $k$-th symmetric power of an endomorphism.

It is known \cite{Fried}, that the Ruelle zeta function can be expressed in terms of Selberg zeta functions. We have
\newline
\begin{equation}\label{trijedan}
Z_{R,\chi}\br{s,\sigma}=\prod_{p=0}^{n-1}\prod_{\br{\tau,\lambda}\in I_{p}}Z_{S,\chi}\br{s+\rho-\lambda,\tau\otimes\sigma}^{\br{-1}^{p}},
\end{equation}
where
\[I_{p}=\set{\br{\tau,\lambda}\,|\,\tau\in \hat{M},\lambda\in\mathbb{R}}\]
\newline
are such sets that $\Lambda^{p}\mathfrak{n}_{\mathbb{C}}$ decomposes with respect to $MA$ as

\[\Lambda^{p}\mathfrak{n}_{\mathbb{C}}=\sum_{\br{\tau,\lambda}\in I_{p}}V_{\tau}\otimes\mathbb{C}_{\lambda}.\]
\newline
Here, $V_{\tau}$ is the space of the representation $\tau$ and $\mathbb{C}_{\lambda}$, $\lambda\in\mathbb{C}$ is the one-dimensional representation of $A$ given by $A\ni a\rightarrow a^{\lambda}$.

Let $d_{Y}=-\br{-1}^{\frac{n}{2}}$. Concerning the singularities of the Selberg zeta function, the following is known.

\begin{teorem}{A.}\label{A}
\cite[p.~113, Th. 3.15]{Bunke} \textit{The Selberg zeta function $Z_{S,\chi}\br{s,\sigma}$ has a
meromorphic continuation to all of $\mathbb{C}$. If $\gamma$ is
$\sigma$-admissible, the singularities of $Z_{S,\chi}\br{s,\sigma}$ are the following:\\
\begin{itemize}
    \item[(1)] at $\pm\II$$s$ of order $m_{\chi}\br{s,\gamma,\sigma}$ if
    $s\neq 0$ is an eigenvalue of
    $A_{Y,\chi}\br{\gamma,\sigma}$,\\
    \item[(2)] at $s=0$ of order $2m_{\chi}\br{0,\gamma,\sigma}$ if
    $0$ is an eigenvalue of $A_{Y,\chi}\br{\gamma,\sigma}$,\\
    \item[(2)] at $-s$, $s\in T\br{\mathbb{N}-\epsilon_{\sigma}}$ of
    order
    $2\frac{d_{Y}\dim\br{\chi}\vol\br{Y}}{\vol\br{X_{d}}}m_{d}\br{s,\gamma,\sigma}$. Then $s>0$ is an eigenvalue of $A_{d}\br{\gamma,\sigma}$.\\
\end{itemize}
If two such points coincide, then the orders add up.}
\end{teorem}\

\section{Main result}

We prove the following theorem.

\begin{theorem}\label{novo}
If $\gamma$ is $\sigma$-admissible, then there exist entire functions $Z_{S}^{1}\br{s}$, $Z_{S}^{2}\br{s}$ of order at most $n$ such that
\begin{equation}\label{jedan}
Z_{S,\chi}\br{s,\sigma}=\frac{Z_{S}^{1}\br{s}}{Z_{S}^{2}\br{s}}.
\end{equation}\
\newline
Here, the zeros of $Z_{S}^{1}\br{s}$ correspond to the
zeros of $Z_{S,\chi}\br{s,\sigma}$ and the zeros of
$Z_{S}^{2}\br{s}$ correspond to the poles of
$Z_{S,\chi}\br{s,\sigma}$. The orders of the zeros of
$Z_{S}^{1}\br{s}$ resp. $Z_{S}^{2}\br{s}$ equal the orders of the
corresponding zeros resp. poles of $Z_{S,\chi}\br{s,\sigma}$.
\end{theorem}
\begin{proof}
Let $N_{1}\br{r}=\#\set{s\in\spec{A_{Y,\chi}\br{\gamma,\sigma}}|\abs{s}\leq r}$. By the Weyl asymptotic law (see, \cite[p.~66]{Bunke}),
\begin{equation}\label{dva}
N_{1}\br{r}\sim C_{1}r^{n},
\end{equation}
as $r\rightarrow +\infty$.

Now, we derive the Weyl asymptotic law for $A_{d}\br{\gamma,\sigma}$. In \cite[p.~529]{AG}, the Weyl asymptotic law for $A_{d}\br{\gamma,\sigma}$ is derived quite quickly by relying on the fact that $A_{d}^{2}\br{\gamma,\sigma}$ is elliptic and of the second order (see, \cite[p.~21]{Bunke2}). In this paper, we follow a more traditional and more informative, step by step, approach of Voros \cite{Voros}. We shall satisfy the assumptions $\br{1.1}$, $\br{1.2}$ and $\br{1.3}$ in \cite[p.~440]{Voros}.

By \cite[p.~109]{Bunke}, $s\in T\br{\mathbb{N}-\epsilon_{\sigma}}$ is an eigenvalue of $A_{d}\br{\gamma,\sigma}$ with multiplicity $m_{d}\br{s,\gamma,\sigma}$. $A_{d}\br{\gamma,\sigma}$ may have more eigenvalues, but the weighted multiplicities of these additional eigenvalues are zero. Moreover, $m_{d}\br{0,\gamma,\sigma}=0$. Since $\epsilon_{\sigma}\in\set{0,\frac{1}{2}}$ and $T\geq 1$, we have $s>0$ for $s\in T\br{\mathbb{N}-\epsilon_{\sigma}}$. Hence, it makes sense to define $N_{2}\br{r}=\#\set{s\in\spec{A_{d}\br{\gamma,\sigma}}|s\leq r}$.

Now, it is easily seen that the eigenvalues of $A_{d}\br{\gamma,\sigma}$ may be arranged in a manner that $\br{1.1}$ in \cite{Voros} holds true.

By \cite[p.~70, Def. 2.1]{Bunke}, the theta function

\[\theta_{d}\br{t,\sigma}=\TRR{e^{-tA_{d}\br{\gamma,\sigma}}}=\sum_{s\in\mathbb{C}}m_{d}\br{s,\gamma,\sigma}e^{-ts}=\]

\[\sum_{s\in T\br{\mathbb{N}-\epsilon_{\sigma}}}m_{d}\br{s,\gamma,\sigma}e^{-ts}\]
\newline
converges for $\RE{\br{t}}>0$. Hence, $\br{1.2}$ in \cite{Voros} is satisfied.

Finally, by \cite[p.~120]{Bunke}, $\theta_{d}\br{t,\sigma}$ admits a full asymptotic expansion

\[\theta_{d}\br{t,\sigma}\sim\sum_{k=-n}^{\infty}d_{k}t^{k}\]
\newline
for $t\rightarrow 0$. In other words, $\br{1.3}$ in \cite{Voros} holds also true.

Now, by $\br{1.4}$ in \cite[p.~440]{Voros}, the eigenvalues of $A_{d}\br{\gamma,\sigma}$ satisfy the Weyl asymptotic law
\begin{equation}\label{tri}
N_{2}\br{r}\sim \frac{d_{-n}}{\Gamma\br{1+n}}r^{n}=C_{2}r^{n},
\end{equation}
as $r\rightarrow +\infty$.

Denote by $S_{1}$ resp. $S_{2}$ the sets consisting of the singularities of $Z_{S,\chi}\br{s,\sigma}$ appearing in $\br{1}$ and $\br{2}$ resp. $\br{3}$ of Theorem A. Reasoning as in \cite[p.~529]{AG}, and using (\ref{dva}), we obtain

\begin{equation}\label{cetiri}
\sum_{s\in
S_{1}\setminus\set{0}}\abs{s}^{-\br{n+\varepsilon}}=O\br{1},
\end{equation}
\newline
for $\varepsilon>0$. Similarly, by (\ref{tri}), we have

\begin{equation}\label{pet}
\sum_{s\in S_{2}}\abs{s}^{-\br{n+\varepsilon}}=O\br{1},
\end{equation}
\newline
for $\varepsilon>0$. Consequently, for $\varepsilon>0$, by (\ref{cetiri})
and (\ref{pet}) we conclude

\begin{equation}\label{sest}
\sum_{s\in
S\backslash\set{0}}\abs{s}^{-\br{n+\varepsilon}}<\infty,
\end{equation}
\newline
where $S$ denotes the set of singularities of
$Z_{S,\chi}\br{s,\sigma}$.

Let $R_{1}$ resp. $R_{2}$ denote the sets of zeros resp. poles of $Z_{S,\chi}\br{s,\sigma}$. It follows from (\ref{sest}) that

\begin{equation}\label{sedam}
\sum_{s\in R_{i}\setminus\set{0}}\abs{s}^{-\br{n+1}}\leq\sum_{s\in R_{i}\setminus\set{0}}\abs{s}^{-\br{n+\varepsilon}}\leq\sum_{s\in S\setminus\set{0}}\abs{s}^{-\br{n+\varepsilon}}<\infty,
\end{equation}
\newline
for $i=1,2$ and $\varepsilon>0$.

In \cite{AG}, the authors followed Park \cite[pp.~93-94]{Park}. The relation (\ref{sedam}) and Theorem 2.6.5 in \cite[p.~19]{Boas} were used to prove that the canonical products $W_{i}\br{s}$ over $R_{i}\backslash\set{0}$, $i=1,2$, are entire functions of order not larger than $n$ over $\mathbb{C}$. Therefore, by using nothing else but the fact that the logarithmic derivative of $Z_{S,\chi}\br{s,\sigma}$ is a Dirichlet series absolutely convergent for $\RE\br{s}\gg 0$, the conclusion was derived that $\degg\br{g\br{s}}\leq n$, where $g\br{s}$ is a polynomial such that $Z_{S,\chi}\br{s,\sigma}W_{1}\br{s}^{-1}W_{2}\br{s}s^{-2m_{\chi}\br{0,\gamma,\sigma}}=e^{g\br{s}}$. This completed the proof in \cite{AG}.

As opposed to \cite{AG}, in this paper we proceed following Pavey \cite[pp.~38-40]{Pav}. At the very beginning (not at the end of the proof), we conclude that $Z_{S,\chi}\br{s,\sigma}$ is of the form (\ref{jedan}). Thus, we reduce the problem to proving the part related to the order of $Z_{S}^{i}\br{s}$, $i=1,2$. We use (\ref{sedam}) and the Weierstrass Factorization Theorem \cite[p.~170]{Con} to form products (\ref{osam}) (not only the canonical products). Reasoning in a similar way as in \cite{AG}, we prove that the canonical products $P_{i}\br{s}$, $i=1,2$, which appear in (\ref{osam}), are entire functions of order at most $n$ over $\mathbb{C}$. In this way we further reduce the problem to proving the fact that $\degg\br{g_{i}\br{s}}\leq n$, $i=1,2$, where $g_{1}\br{s}$ and $g_{2}\br{s}$ are entire functions which naturally arise from the factorization theorem. We obtain two representations for the $n$-th derivative of the logarithmic derivative of $Z_{S,\chi}\br{s,\sigma}$ and compare them. The first one, (\ref{devet}), follows from (\ref{jedan}) and (\ref{osam}), while the second one (\ref{petnaest}), stems from the representation (\ref{deset}) of the Selberg zeta function $Z_{S,\chi}\br{s,\sigma}$ in terms of regularized determinants of elliptic operators $A_{Y,\chi}\br{\gamma,\sigma}$ and $A_{d}\br{\gamma,\sigma}$. We complete the proof by concluding that $\degg\br{g_{1}\br{s}-g_{2}\br{s}}\leq n$.

Now, we proceed with the proof.

By Theorem A, $Z_{S,\chi}\br{s,\sigma}$ is meromorphic. Hence, (see, \cite[p.~38]{Pav}), it may be represented in the form (\ref{jedan}), where $Z_{S}^{i}\br{s}$, $i=1,2$ are entire functions satisfying assumptions of Theorem \ref{novo}, except possibly the part concerning the order of $Z_{S}^{i}\br{s}$, $i=1,2$. In other words, it remains to prove that $Z_{S}^{i}\br{s}$, $i=1,2$ are of order at most $n$.

Now, $R_{i}$ is the set of zeros of $Z_{S}^{i}\br{s}$, $i=1,2$.

By (\ref{sedam}) and \cite[p.~282]{Con}, $Z_{S}^{i}\br{s}$ is of finite rank $p_{i}\leq n$ for $i=1,2$. According to the Weierstrass Factorization Theorem \cite[p.~170]{Con} (see also \cite[p.~39]{Pav}), there exist entire functions $g_{i}\br{s}$, $i=1,2$ such that

\begin{equation}\label{osam}
Z_{S}^{i}\br{s}=s^{n_{i}}e^{g_{i}\br{s}}\prod_{\rho\in R_{i}\setminus\set{0}}\br{1-\frac{s}{\rho}}\exp\br{\frac{s}{\rho}+\frac{s^{2}}{2\rho^{2}}+...+\frac{s^{p_{i}}}{p_{i}\rho^{p_{i}}}}=
\end{equation}

\[s^{n_{i}}e^{g_{i}\br{s}}P_{i}\br{s},\]
\newline
where $n_{i}$ is the order of the zero of $Z_{S}^{i}\br{s}$ at $s=0$, $i=1,2$.

$P_{i}\br{s}$ is a canonical product of rank $p_{i}$, $i=1,2$. Denote by $q_{i}$ the exponent of convergence of $R_{i}\setminus\set{0}$, $i=1,2$ (see, \cite[p.~286]{Con}). By (\ref{sedam}), $q_{i}\leq n$, $i=1,2$. Now, by \cite[p.~287, (d)]{Con}, the order of $P_{i}\br{s}$ is $q_{i}\leq n$, $i=1,2$.

Hence, in order to complete the proof of the theorem, it is enough to prove that $\degg\br{g_{i}\br{s}}\leq n$ for $i=1,2$.

By taking logarithms of both sides in (\ref{jedan}) and having in mind (\ref{osam}), we obtain
\[\log Z_{S,\chi}\br{s,\sigma}=\br{g_{1}\br{s}-g_{2}\br{s}}+\]
\[\br{n_{1}\log s+\sum_{\rho\in R_{1}\setminus\set{0}}\log\br{\rho-s}-\sum_{\rho\in R_{1}\setminus\set{0}}\log\rho}-\]
\[\br{n_{2}\log s+\sum_{\rho\in R_{2}\setminus\set{0}}\log\br{\rho-s}-\sum_{\rho\in R_{2}\setminus\set{0}}\log\rho}+\]
\[\sum_{\rho\in R_{1}\setminus\set{0}}\br{\frac{s}{\rho}+\frac{s^{2}}{2\rho^{2}}+...+\frac{s^{p_{1}}}{p_{1}\rho^{p_{1}}}}-
\sum_{\rho\in R_{2}\setminus\set{0}}\br{\frac{s}{\rho}+\frac{s^{2}}{2\rho^{2}}+...+\frac{s^{p_{2}}}{p_{2}\rho^{p_{2}}}}.\]
\newline
Differentiating $n+1$ times and taking into account the fact that $p_{i}\leq n$ for $i=1,2$ as well as the fact that $n$ is even, we conclude that

\begin{equation}\label{devet}
\frac{d^{n}}{ds^{n}}\frac{Z_{S,\chi}^{'}\br{s,\sigma}}{Z_{S,\chi}\br{s,\sigma}}=\frac{d^{n+1}}{ds^{n+1}}\br{g_{1}\br{s}-g_{2}\br{s}}+
\end{equation}
\[n_{1}\frac{n!}{s^{n+1}}-n_{2}\frac{n!}{s^{n+1}}+\sum_{\rho\in R_{1}\setminus\set{0}}\frac{n!}{\br{s-\rho}^{n+1}}-\sum_{\rho\in R_{2}\setminus\set{0}}\frac{n!}{\br{s-\rho}^{n+1}}.\]
\newline
On the other hand, by \cite[p.~118, Th. 3.19]{Bunke},

\begin{equation}\label{deset}
Z_{S,\chi}\br{s,\sigma}=
\end{equation}
\[\exp\br{\frac{\dim\br{\chi}\chi\br{Y}}{\chi\br{X_{d}}}
\sum\limits_{m=1}^{\frac{n}{2}}c_{-m}\frac{s^{2m}}{m!}\br{\sum\limits_{r=1}^{m-1}\frac{1}{r}-2\sum\limits_{r=1}^{2m-1}\frac{1}{r}}}\cdot\]
\[\det\br{A_{Y,\chi}\br{\gamma,\sigma}^{2}+s^{2}}
\det\br{A_{d}\br{\gamma,\sigma}+s}^{-\frac{2\dim\br{\chi}\chi\br{Y}}{\chi\br{X_{d}}}}=\]
\[e^{h_{3}\br{s}}\det\br{A_{Y,\chi}\br{\gamma,\sigma}^{2}+s^{2}}
\det\br{A_{d}\br{\gamma,\sigma}+s}^{-\frac{2\dim\br{\chi}\chi\br{Y}}{\chi\br{X_{d}}}},\]
\newline
where the coefficients $c_{k}$ are defined by the asymptotic
expansion

\[\TRR{e^{-tA_{d}\br{\gamma,\sigma}^{2}}}\overset{t\rightarrow 0}{\sim}\sum\limits_{k=-\frac{n}{2}}^{\infty}c_{k}t^{k}.\]
By \cite[p.~36, (1.13)]{Bunke},
\[\frac{\vol\br{Y}}{\vol\br{X_{d}}}=\br{-1}^{\frac{n}{2}}\frac{\chi\br{Y}}{\chi\br{X_{d}}}.\]
Hence,
\[2\frac{d_{Y}\dim\br{\chi}\vol\br{Y}}{\vol\br{X_{d}}}=-2\br{-1}^{\frac{n}{2}}\dim\br{\chi}\br{-1}^{\frac{n}{2}}\frac{\chi\br{Y}}{\chi\br{X_{d}}}=\]
\[-\frac{2\dim\br{\chi}\chi\br{Y}}{\chi\br{X_{d}}}.\]
Now, (\ref{deset}) becomes
\begin{equation}\label{jedanaest}
Z_{S,\chi}\br{s,\sigma}=
\end{equation}
\[e^{h_{3}\br{s}}\det\br{A_{Y,\chi}\br{\gamma,\sigma}^{2}+s^{2}}\det\br{A_{d}\br{\gamma,\sigma}+s}^{2\frac{d_{Y}\dim\br{\chi}\vol\br{Y}}{\vol\br{X_{d}}}}.\]
\newline
By \cite[pp.~120-121]{Bunke}, (see also \cite{Brocker})

\[\det\br{A_{d}\br{\gamma,\sigma}+s}=e^{h_{1}\br{s}}\Delta^{-}\br{s},\]
\newline
where $h_{1}\br{s}$ is a polynomial of order $n$ and $\Delta^{-}\br{s}$ is the Weierstrass product

\[\Delta^{-}\br{s}=\prod_{\mu\in\spec{A_{d}\br{\gamma,\sigma}}}\SqBr{\br{1+\frac{s}{\mu}}\exp\br{\sum_{r=1}^{n}\frac{\br{-s}^{r}}{r\mu^{r}}}}^{m_{d}\br{\mu,\gamma,\sigma}}=\]

\[\prod_{\mu\in\spec{A_{d}\br{\gamma,\sigma}}}\SqBr{\br{1-\frac{s}{-\mu}}\exp\br{\sum_{r=1}^{n}\frac{s^{r}}{r\br{-\mu}^{r}}}}^{m_{d}\br{\mu,\gamma,\sigma}}.\]
\newline
Hence, by the definition of the set $S_{2}$,

\begin{equation}\label{dvanaest}
\det\br{A_{d}\br{\gamma,\sigma}+s}^{2\frac{d_{Y}\dim\br{\chi}\vol\br{Y}}{\vol\br{X_{d}}}}=e^{2\frac{d_{Y}\dim\br{\chi}\vol\br{Y}}{\vol\br{X_{d}}}h_{1}\br{s}}\cdot
\end{equation}

\[\prod_{\mu\in\spec{A_{d}\br{\gamma,\sigma}}}\SqBr{\br{1-\frac{s}{-\mu}}
\exp\br{\sum_{r=1}^{n}\frac{s^{r}}{r\br{-\mu}^{r}}}}^{2\frac{d_{Y}\dim\br{\chi}\vol\br{Y}}{\vol\br{X_{d}}}m_{d}\br{\mu,\gamma,\sigma}}=\]

\[e^{H_{1}\br{s}}\prod_{\rho\in S_{2}}\br{1-\frac{s}{\rho}}
\exp\br{\sum_{r=1}^{n}\frac{s^{r}}{r\rho^{r}}}.\]
\newline
Similarly, one obtains
\begin{equation}\label{trinaest}
\det\br{A_{Y,\chi}\br{\gamma,\sigma}^{2}+s^{2}}=
\end{equation}

\[s^{M}e^{h_{2}\br{s}}\prod_{\rho\in S_{1}\setminus\set{0}}\br{1-\frac{s}{\rho}}\exp\br{\sum_{r=1}^{n}\frac{s^{r}}{r\rho^{r}}},\]
\newline
where $h_{2}\br{s}$ is a polynomial of order $n$ and $M=2m_{\chi}\br{0,\gamma,\sigma}$ if $0$ is an eigenvalue of $A_{Y,\chi}\br{\gamma,\sigma}$. Otherwise, $M=0$.

Combining (\ref{jedanaest}), (\ref{dvanaest}) and (\ref{trinaest}), we get

\begin{equation}\label{cetrnaest}
Z_{S,\chi}\br{s,\sigma}=s^{M}e^{H\br{s}}\cdot
\end{equation}
\[\prod_{\rho\in S_{1}\setminus\set{0}}\br{1-\frac{s}{\rho}}\exp\br{\sum_{r=1}^{n}\frac{s^{r}}{r\rho^{r}}}\prod_{\rho\in S_{2}}\br{1-\frac{s}{\rho}}
\exp\br{\sum_{r=1}^{n}\frac{s^{r}}{r\rho^{r}}},\]
\newline
where $H\br{s}=h_{3}\br{s}+h_{2}\br{s}+H_{1}\br{s}$. Obviously, $\deg\br{H\br{s}}=n$.

Hence, reasoning as in the derivation of (\ref{devet}), we deduce

\begin{equation}\label{petnaest}
\frac{d^{n}}{ds^{n}}\frac{Z_{S,\chi}^{'}\br{s,\sigma}}{Z_{S,\chi}\br{s,\sigma}}=
\end{equation}
\[M\frac{n!}{s^{n+1}}+\sum_{\rho\in S_{1}\setminus\set{0}}\frac{n!}{\br{s-\rho}^{n+1}}+\sum_{\rho\in S_{2}}\frac{n!}{\br{s-\rho}^{n+1}}.\]
\newline
Combining (\ref{devet}) and (\ref{petnaest}), we obtain

\[\frac{d^{n+1}}{ds^{n+1}}\br{g_{1}\br{s}-g_{2}\br{s}}+\]

\[n_{1}\frac{n!}{s^{n+1}}-n_{2}\frac{n!}{s^{n+1}}+\sum_{\rho\in R_{1}\setminus\set{0}}\frac{n!}{\br{s-\rho}^{n+1}}-\sum_{\rho\in R_{2}\setminus\set{0}}\frac{n!}{\br{s-\rho}^{n+1}}=\]

\[M\frac{n!}{s^{n+1}}+\sum_{\rho\in S_{1}\setminus\set{0}}\frac{n!}{\br{s-\rho}^{n+1}}+\sum_{\rho\in S_{2}}\frac{n!}{\br{s-\rho}^{n+1}}.\]
\newline
From the definitions of the sets $R_{i}$, $S_{i}$, $i=1,2$, it follows that the corresponding sums cancel each other, i.e., we have

\[\frac{d^{n+1}}{ds^{n+1}}\br{g_{1}\br{s}-g_{2}\br{s}}+n_{1}\frac{n!}{s^{n+1}}-n_{2}\frac{n!}{s^{n+1}}=M\frac{n!}{s^{n+1}}.\]
\newline
Recall the definitions of $M$, $n_{i}$, $i=1,2$.

If $0$ is not an eigenvalue of $A_{Y,\chi}\br{\gamma,\sigma}$ then $s=0$ is not the singularity of $Z_{S,\chi}\br{s,\sigma}$. Hence, $M=n_{1}=n_{2}=0$. Otherwise, $s=0$ is either a zero or a pole of $Z_{S,\chi}\br{s,\sigma}$ and $M=2m_{\chi}\br{0,\gamma,\sigma}$. In the first case, $n_{1}=2m_{\chi}\br{0,\gamma,\sigma}$, $n_{2}=0$. Else, $n_{1}=0$, $n_{2}=-2m_{\chi}\br{0,\gamma,\sigma}$. In other words,
\[n_{1}\frac{n!}{s^{n+1}}-n_{2}\frac{n!}{s^{n+1}}=M\frac{n!}{s^{n+1}}.\]
Hence,
\begin{equation}\label{sesnaest}
\frac{d^{n+1}}{ds^{n+1}}\br{g_{1}\br{s}-g_{2}\br{s}}=0.
\end{equation}
\newline
Therefore, $g_{1}\br{s}-g_{2}\br{s}$ is a polynomial of degree at most $n$. Note that this fact does not necessarily imply that $\degg\br{g_{i}\br{s}}\leq n$ for $i=1,2$. Namely, (\ref{sesnaest}) only implies that $g_{1}\br{s}$ and $g_{2}\br{s}$ are of the form

\[g_{1}\br{s}=\sum_{i>n}\alpha_{i}s^{i}+\sum_{i=0}^{n}\beta_{i}s^{i},\]

\[g_{2}\br{s}=\sum_{i>n}\alpha_{i}s^{i}+\sum_{i=0}^{n}\delta_{i}s^{i},\]
\newline
for some coefficients $\alpha_{i}$, $\beta_{i}$, $\delta_{i}$. This, in turn, together with (\ref{osam}) and (\ref{jedan}) allow us to assume that $\degg\br{g_{i}\br{s}}\leq n$ for $i=1,2$. This completes the proof.
\end{proof}\

A trivial consequence of the Theorem \ref{novo} and the relation (\ref{trijedan}) is the following corollary (see also, \cite[p.~530]{AG})
\newline
\begin{corollary}\label{novo1}
The Ruelle zeta function $Z_{R,\chi}\br{s,\sigma}$ can be expressed as

\[Z_{R,\chi}\br{s,\sigma}=\frac{Z_{R}^{1}\br{s}}{Z_{R}^{2}\br{s}}.\]
\newline
Here, $Z_{R}^{1}\br{s}$, $Z_{R}^{2}\br{s}$ are entire functions of order at most $n$.
\end{corollary}\

\begin{rem}{}\label{acknowledgment}
The authors thank the referee for valuable advices and suggestions.
\end{rem}

\end{document}